\documentclass[reqno]{amsart}%
\usepackage{amsfonts, amsthm}
\usepackage{eurosym}
\usepackage{amsmath}
\usepackage{amssymb}
\usepackage{graphicx}
\usepackage{xcolor}
\usepackage{amsfonts}
\usepackage{enumitem}
\usepackage{subcaption}
\usepackage{float}
\usepackage{caption}%
\providecommand{\U}[1]{\protect\rule{.1in}{.1in}}
\providecommand{\U}[1]{\protect\rule{.1in}{.1in}}
\newtheorem{theorem}{Theorem}[section]
\theoremstyle{plain}

\newtheorem{lemma}[theorem]{Lemma}

\newtheorem{remark}{Remark}[section]
\numberwithin{equation}{section}

\begin{document}
\title[$L^{p}$-Convergence and Denoising of Max-Min NN Operators]{On the $L^{p}$-Convergence and Denoising Performance of Durrmeyer-Type Max-Min
Neural Network Operators}
\author[BERKE \c{S}AH\.{I}N, \.{I}SMA\.{I}L ASLAN]{BERKE \c{S}AH\.{I}N, \.{I}SMA\.{I}L ASLAN${^{\ast}}$}
\thanks{$^{\ast}$Corresponding author.}

\begin{abstract}
In this paper, we investigate Durrmeyer-type generalizations of
maximum-minimum neural network operators. The primary objective of this study
is to establish the convergence of these operators in the $L^{p}$ norm for
functions $f\in L^{p}([a,b],[0,1])$ with $1\leq p<\infty$. To this end, we
analyze the properties of sigmoidal functions and maximum-minimum operations,
subsequently establishing the convergence of the proposed operator in
pointwise, supremum, and $L^{p}$ norms. Furthermore, we derive quantitative
estimates for the rates of convergence. In the applications section, numerical
and graphical examples demonstrate that the proposed Durrmeyer-type operators
provide smoother approximations compared to Kantorovich-type and standard
max-min operators. Finally, we highlight the superior filtering performance of
these operators in signal analysis, validating their effectiveness in both
approximation and data processing tasks.

\end{abstract}
\maketitle

\section{Introduction}

In approximation theory, neural network-based operators have gained
significant importance in both theoretical and applied fields, particularly
over the past decade, and have been extensively studied \cite{
anas7,is,aslan2,durrmeyer, cost1, cost25, cheang, cao,cao2}. Initial studies
primarily focused on linear neural network operators, and the convergence
properties of these operators in various function spaces were examined in
detail \cite{costez,cost1,coroianu4,kad}. For instance, Costarelli and Spigler
in \cite{cost1} introduced the following linear positive neural network
operators activated by sigmoidal functions and demonstrated that these
operators exhibit proper convergence properties for continuous functions.%

\begin{equation}
F_{n}\left(  f;x\right)  :=\dfrac{\sum\limits_{k=\left\lceil na\right\rceil
}^{\left\lfloor nb\right\rfloor }f\left(  \frac{k}{n}\right)  \phi_{\sigma
}\left(  nx-k\right)  }{\sum\limits_{d=\left\lceil na\right\rceil
}^{\left\lfloor nb\right\rfloor }\phi_{\sigma}\left(  nx-d\right)  }\text{
\ \ (}x\in\left[  a,b\right]  \text{)}, \label{1}%
\end{equation}
where $\phi_{\sigma}$ is an appropriate linear combination of activation
functions $\sigma:\mathbb{R}\rightarrow\mathbb{R}$ and $f:\left[  a,b\right]
\rightarrow\mathbb{R}$ is a bounded function. In the above definition,
$\left\lfloor \cdot\right\rfloor $ and $\left\lceil \cdot\right\rceil $
represent the floor and the ceiling function for a given number. However,
certain limitations of linear operators have led researchers
\cite{anas7,is,cost25} to explore nonlinear forms. In this context, the
concept of pseudo-linear operators, first introduced by Bede et al
\cite{bede1, bede2}, has been developed by redefining the operations of
addition and multiplication in a different algebraic structure. They
demonstrated that these operators often outperform classical approximation
operators that rely on standard sum and product operations. These studies have
formed the basis for much of the contemporary research on pseudo-linear
operators, leading to numerous publications in this field \cite{bede3,
bede4,bede5,coroianu1,coroianu2,coroianu3,duman,aslan,gokcer2,holhos1}.
Costarelli and Vinti in \cite{cost2}, on the other hand, introduced
max-product variants of neural network operators and proved the uniform
convergence properties of these operators.

Later, in \cite{is} Aslan examined the convergence properties of max-min
neural network operators with sigmoidal activation functions, addressing their
convergence, order of approximation, and comparative performance with other
operators. It has been demonstrated that the max-min variant provides a better
approximation compared to linear operators in certain cases. Notably, max-min
operators exhibit a nonlinear and non-homogeneous structure. Furthermore, due
to the nature of the maximum operation, computation times are significantly
lower compared to linear operators. They are also effective in processing
noisy signals and exhibit better performance compared to max-product versions
(see \cite{is}). The max-min neural network operator studied in \cite{is} is
given as follows :%

\begin{equation}
F_{n}^{(m)}(f;x):=\bigvee\limits_{k=\left\lceil na\right\rceil }^{\left\lfloor
nb\right\rfloor }f\left(  \frac{k}{n}\right)  \wedge\frac{\phi_{\sigma}\left(
nx-k\right)  }{\bigvee\limits_{d=\left\lceil na\right\rceil }^{\left\lfloor
nb\right\rfloor }\phi_{\sigma}\left(  nx-d\right)  }\text{ \ \ (}x\in\left[
a,b\right]  \text{)},\label{4}%
\end{equation}
such that $\bigvee\nolimits_{i=1}^{k}x_{i}:=\max\nolimits_{i=1,\cdots
,k}\{x_{i}\}$ and $A\wedge B=\min\{A,B\}$. Eventually, in \cite{aslan2} Aslan,
De Marchi and Erb introduced the Kantorovich variant of (\ref{4}). The study
establishes their $L^{p}$-convergence for $1\leq p<\infty$, provides
quantitative estimates for the rate of approximation, and includes numerical
experiments illustrating their performance on discontinuous and noisy
functions. Moreover, a comparative analysis demonstrates that the Kantorovich
max-min operators outperform several classical neural network operators,
especially in approximating biomedical signals such as ECG data. The operator
is presented as follows,
\[
K_{n}^{(m)}(f;x):=\bigvee\limits_{k=\left\lceil na\right\rceil }^{\left\lfloor
nb\right\rfloor -1}n\int\limits_{\frac{k}{n}}^{\frac{k+1}{n}}f(u)du\wedge
\frac{\phi_{\sigma}\left(  nx-k\right)  }{\bigvee\limits_{d=\left\lceil
na\right\rceil }^{\left\lfloor nb\right\rfloor -1}\phi_{\sigma}\left(
nx-d\right)  }\text{ \ \ (}x\in\left[  a,b\right]  \text{)}.
\]
However, Durrmeyer-type generalizations of neural network operators have not
yet been sufficiently investigated in the literature (\cite{durrmeyer,wu}).
Durrmeyer variants provide operators with a more general integral-based
structure, enabling convergence results to be obtained, particularly in
weighted $L^{p}$-spaces and broader function classes.

Let $f:$ $[a,b]\rightarrow%
\mathbb{R}
$ be a bounded and $L_{1}$-integrable function. Then, for $x\in\left[
a,b\right]  $, the linear Durrmeyer-type neural network operators studied in
\cite{durrmeyer} are given by
\begin{equation}
D_{n}\left(  f;x\right)  :=\frac{\sum\limits_{k=na}^{nb-1}\left[
n\int\limits_{a}^{b}\chi(nt-k)f(t)dt\right]  \phi_{\sigma}\left(  nx-k\right)
}{\sum\limits_{k=na}^{nb-1}\left[  n\int\limits_{a}^{b}\chi(nt-k)dt\right]
\phi_{\sigma}\left(  nx-k\right)  }\label{3}%
\end{equation}
such that $n\in%
\mathbb{N}
^{+}$, where $\chi:\mathbb{R}\rightarrow\lbrack0,\infty)$ satisfies some basic
conditions. Motivated by the aforementioned operator and the studies in
\cite{is,aslan2}, we investigate the nonlinear and non-homogeneous variants of
these operators. The primary objective of this study is to introduce a
Durrmeyer-type generalization of Kantorovich-type max-min neural network
operators and to thoroughly examine their convergence properties in $L^{p}%
$-spaces. In this context, this article will first introduce Durrmeyer-type
max-min neural network operators and investigate their fundamental properties.
Durrmeyer operators provide a stronger smoothing effect compared to their
Kantorovich versions, which makes them particularly suitable for approximation
processes requiring higher regularity. Subsequently, the convergence behavior
of these operators in $L^{p}$-spaces will be analyzed, and quantitative
estimates for their convergence rates will be derived.

In the application part of our study, by selecting appropriate kernel
functions, it will be visually demonstrated that the Durrmeyer-type max-min
neural network operators provide a smoother approximation compared to both the
max-min Kantorovich and the standard max-min forms. Moreover, it is known that
max-min type neural network operators exhibit a better filtering effect in
signal analysis than max-product operators \cite{is}. In this study, we also
show that Durrmeyer-type max-min neural network operators possess superior
filtering properties compared to their Kantorovich-type counterparts.
Additionally, we analyze the filtering properties of our operator on speech
signals contaminated by salt and pepper noise.

\section{Preliminaries on Neural Networks and Pseudo-Linear Operators}

In this section, we recall below the main notions and notations related to NN
operators, which will be used throughout this section.

Let $\chi: \mathbb{R} \to[0,\infty)$ be a bounded and $L_{1}$-integrable
function such that\\ $\int_{0}^{1}\chi(u)du = \mathcal{A} > 0$ and its discrete
absolute moment of order zero is finite, that is,%

\[
M_{0}(\chi):=\sup_{t\in\mathbb{R}}\sum\limits_{k\in%
\mathbb{Z}
}\chi(t-k)<\infty.
\]

In addition, for any $\upsilon>0$ the continuous absolute moment of order
$\upsilon$ is defined as follows,%

\[
\tilde{M}_{\upsilon}(\chi):=\int\limits_{\mathbb{R}}\chi(u)|u|^{\upsilon}dt,
\]
where we assume in this paper that $\tilde{M}_{1}(\chi)<\infty$.

Let $\sigma:\mathbb{R}\rightarrow\mathbb{R}$ be given. Then $\sigma$ is called
a sigmoidal function if $\lim_{x\rightarrow-\infty}\sigma\left(  x\right)  =0$
and $\lim_{x\rightarrow\infty}\sigma\left(  x\right)  =1.$ For the rest of the
paper, we assume that $\sigma$ is a nondecreasing sigmoidal function. We also
assume that $\sigma\left(  3\right)  >\sigma\left(  1\right)  $, which is a
minor restriction put in place to prevent certain technical issues.

Together with above definitions, the following conditions are also needed.

\begin{enumerate}
\item[$\left(  \Sigma1\right)  $] $\sigma\left(  x\right)  -1/2$ is an odd
function on the real line,

\item[$\left(  \Sigma2\right)  $] $\sigma\in C^{2}\left(  \mathbb{R}\right)  $
is concave for all $x\in\mathbb{R}_{0}^{+},$

\item[$\left(  \Sigma3\right)  $] $\sigma\left(  x\right)  =O(\left\vert
x\right\vert ^{-(1+\alpha)})$ as $x\rightarrow-\infty$ for some $\alpha>0.$
\end{enumerate}

For the rest of this paper, we use the \textquotedblleft$\alpha$%
\textquotedblright\ symbol exclusively in connection to the condition
specified in $\left(  \Sigma3\right)  $. The kernel $\phi_{\sigma}$ (also
called \textquotedblleft centered bell shaped function\textquotedblright\ in
\cite{carda})\ in the definition of the neural network operator is given by%
\[
\phi_{\sigma}\left(  x\right)  :=\frac{1}{2}\left(  \sigma\left(  x+1\right)
-\sigma\left(  x-1\right)  \right)  \text{ \ for all }x\in\mathbb{R}\text{.}%
\]
Additionally, we can also say that $\phi_{\sigma}$ does not have to be
compactly supported. We need some additional definitions and assumptions. From
the definitions and assumptions given above, we can derive the following
properties for $\phi_{\sigma}$ (see\cite{cost1}).

\begin{lemma}
\label{lemma1}

\begin{enumerate}
\item $\phi_{\sigma}\left(  x\right)  \geq0$ $(\forall x\in\mathbb{R})$ and
$\phi_{\sigma}\left(  2\right)  >0$,

\item $\lim_{x\rightarrow\pm\infty}\phi_{\sigma}\left(  x\right)  =0$,

\item $\phi_{\sigma}$ is nondecreasing if $x<0$ and nonincreasing if $x\geq0$
$($therefore $\phi_{\sigma}\left(  0\right)  \geq\phi_{\sigma}\left(
x\right)  $ $\forall x\in\mathbb{R)},$

\item $\phi_{\sigma}\left(  x\right)  =O(\left\vert x\right\vert ^{-\left(
1+\alpha\right)  })$ as $x\rightarrow\pm\infty$, i.e., there exist two
constants $A,B>0$ such that $\phi_{\sigma}\left(  x\right)  \leq A\left\vert
x\right\vert ^{-\left(  1+\alpha\right)  }$ whenever $\left\vert x\right\vert
>B.$

\item $\phi_{\sigma}\left(  x\right)  $ is an even function.

\item $\forall x\in\mathbb{R},$ $\sum\limits_{k\in\mathbb{Z}}\phi_{\sigma
}\left(  x-k\right)  =1.$
\end{enumerate}
\end{lemma}

From $(3)$ and $(6)$ of Lemma \ref{lemma1}, we can conclude that $\phi
_{\sigma}\in L^{1}\left(  \mathbb{R}\right)  $ (see also\ Remark 1 in
\cite{cost25}). In addition, by the definition of $\sigma$ and $\phi_{\sigma}%
$, we have $\phi_{\sigma}\left(  x\right)  \leq1/2$ for all $x\in\mathbb{R}$.

\begin{remark}
\label{rem0} We reminde that, as in \cite{cost1, cost2}, the condition
$\left(  \Sigma2\right)  $ was only used to prove property $(3)$ of Lemma
\ref{lemma1}. Therefore, the reader can conclude that our theory can also
apply to the non-smooth sigmoid functions, which satisfy the property $(3)$ in
Lemma \ref{lemma1} in place of the condition $\left(  \Sigma2\right)  $.
\end{remark}

The next lemma is necessary for the well-definiteness of Durrmeyer-type
max-min NN operators given in (\ref{a}).

\begin{lemma}
(see \cite{cost2})\label{lemma2}

\begin{enumerate}
\item For a given $x\in\mathbb{R}$,
\[
\bigvee\limits_{k\in\mathbb{Z}}\phi_{\sigma}\left(  nx-k\right)  \geq
\phi_{\sigma}\left(  2\right)  >0
\]
holds for all $n\in\mathbb{N},$

\item Let the interval $\left[  a,b\right]  $ be given. Then for all
$x\in\left[  a,b\right]  $
\[
\bigvee\limits_{k=na}^{nb-1}\phi_{\sigma}\left(  nx-k\right)  \geq\phi
_{\sigma}\left(  2\right)  >0
\]
holds for sufficiently large $n\in\mathbb{N}$.
\end{enumerate}
\end{lemma}

We note that if the index set of maximum operation is infinite, then
\textquotedblleft$\vee$\textquotedblright\ represents the supremum operation.

\begin{lemma}
(see \cite{cost2})\label{lemma0} For each $\delta>0,$ we have%
\[
\bigvee\limits_{\overset{k\in\mathbb{Z}}{\left\vert x-k\right\vert >n\delta}%
}\phi_{\sigma}\left(  x-k\right)  =O\big(  n^{-\left(  1+\alpha\right)
}\big)  \text{ as }n\rightarrow\infty
\]
\newline uniformly in $x\in\mathbb{R}$.
\end{lemma}

Now, let us look at some features of maximum and minimum operations that we
will use frequently in the following section.

\begin{lemma}
(see \cite{bede5})\label{lemma3} If $\bigvee\limits_{i\in\mathbb{Z}}%
x_{i}<\infty$ or $\bigvee\limits_{i\in\mathbb{Z}}y_{i}<\infty$, then
\[
\biggl| \bigvee\limits_{i\in\mathbb{Z}}x_{i}-\bigvee\limits_{i\in\mathbb{Z}%
}y_{i} \biggr| \leq\bigvee\limits_{i\in\mathbb{Z}} |x_{i}-y_{i}|.
\]

\end{lemma}

\begin{lemma}
(see \cite{bede1})\label{lemma4} $\left\vert a\wedge b-a\wedge c\right\vert
\leq a\wedge\left\vert b-c\right\vert $ for all $a,b,c\in\left[  0,1\right]  $
\end{lemma}

\begin{lemma}
$\forall a,b,c\geq0,$ the following inequalities hold
\[
(a+b)\wedge c\leq a\wedge c+b\wedge c
\]

and%
\[
(a+b)\vee c\leq a\vee c+b\vee c.
\]

\end{lemma}

\begin{lemma}
(see \cite{bede1})\label{lemma6} For any $x_{i}\geq0$ and $k\geq0,$
\[
\bigl(  \bigvee\limits_{i\in\mathbb{Z}}x_{i}\bigl)  ^{k}=\bigvee
\limits_{i\in\mathbb{Z}}x_{i}^{k}%
\]
and
\[
\bigl(  \bigwedge\limits_{i\in\mathbb{Z}}x_{i}\bigl)  ^{k}=\bigwedge
\limits_{i\in\mathbb{Z}}x_{i}^{k}.
\]

\end{lemma}

\section{Durrmeyer-type Max-Min Neural Network Operators}

For simplicity, let us assume that $a$ and $b$ are integers with $a < b$. Let
$f: [a, b] \to[0, 1]$ be a bounded and $L_{1}$-integrable function. The
max-min modification of the Durrmeyer-type NN operators given in (\ref{3}) is
introduced as follows
\begin{equation}
D_{n}^{\left(  m\right)  }\left(  f;x\right)  :=\bigvee\limits_{k=na}%
^{nb-1}\frac{\int\limits_{a}^{b}\chi(nt-k)f(t)dt}{\int\limits_{a}^{b}%
\chi(nt-k)dt}\wedge\dfrac{\phi_{\sigma}\left(  nx-k\right)  }{\bigvee
\limits_{d=na}^{nb-1}\phi_{\sigma}\left(  nx-d\right)  }\text{ \ \ (}%
x\in\left[  a,b\right]  \text{)}. \label{a}%
\end{equation}
It is worth noting that although the function $f$ appears to be restricted to the $[0,1]$ interval, this limitation can be addressed via a simple trick. Consequently, the subsequent approximations can be extended to any bounded function $f$ (see \cite{is}). \\ \medskip

Before establishing that the operator is well-defined, we present two
auxiliary inequalities that will be used frequently.
\begin{align}
\int\limits_{a}^{b}\chi(nt-k)dt  &  =\frac{1}{n}\int\limits_{na-k}^{nb-k}%
\chi(u)du\geq\frac{1}{n}\int\limits_{0}^{1}\chi(u)du=\frac{\mathcal{A}}%
{n}>0\label{*}\\
\int\limits_{a}^{b}\chi(nt-k)dt  &  =\frac{1}{n}\int\limits_{na-k}^{nb-k}%
\chi(u)du\leq\frac{1}{n}\int\limits_{\mathbb{R}}\chi(u)du=\frac{||\chi||_{1}%
}{n} \label{**}%
\end{align}
for all $k\in\{na,\ldots,nb-1\}$, where $||\cdot||_{1}$ denotes the $L_{1}$-norm.

From Lemma \ref{lemma2} and (\ref{*})\ the denominator part of the operator is
strictly positive. On the other hand, since $0\leq f\leq1$ then
\[
|D_{n}^{\left(  m\right)  }\left(  f;x\right)  |\leq1
\]
for all $x\in[a,b]$

\begin{lemma}
\label{lemma5} Let $f,g:\left[  a,b\right]  \rightarrow\left[  0,1\right]  $
be two measurable functions.

\begin{enumerate}
\item[(a)] If $\sigma$ is continuous function on $\mathbb{R}$, then
$D_{n}^{\left(  m\right)  }\left(  f\right)  $ is continuous on $\left[
a,b\right]  $,

\item[(b)] If $f\left(  x\right)  \leq g\left(  x\right)  $ for all
$x\in\left[  a,b\right]  ,$ then $D_{n}^{\left(  m\right)  }\left(
f;x\right)  \leq D_{n}^{\left(  m\right)  }\left(  g;x\right)  $ for all
$x\in\left[  a,b\right]  $,

\item[(c)] $D_{n}^{\left(  m\right)  }$ is sublinear, that is, $D_{n}^{\left(
m\right)  }\left(  f+g;x\right)  \leq D_{n}^{\left(  m\right)  }\left(
f;x\right)  +D_{n}^{\left(  m\right)  }\left(  g;x\right)  $ for all
$x\in\left[  a,b\right]  $,

\item[(d)] $\left\vert D_{n}^{\left(  m\right)  }\left(  f;x\right)
-D_{n}^{\left(  m\right)  }\left(  g;x\right)  \right\vert \leq D_{n}^{\left(
m\right)  }\left(  \left\vert f-g\right\vert ;x\right)  $ for all $x\in\left[
a,b\right]  $.
\end{enumerate}
\end{lemma}

The proof can be easily obtained from the definition of the operator and the
lemmas above.

Now, for a fixed $\delta>0,$ $x\in\left[  a,b\right]  $ and $n\in\mathbb{N}$,
we define $B_{\delta,n}\left(  x\right)  $ by
\[
B_{\delta,n}\left(  x\right)  :=\big\{  k=na,na+1,\ldots,nb-1:\big|
x-\frac{k}{n}\big| >\delta\big\}  \text{.}%
\]
Our approximation theorems now read as follows.

\begin{theorem}
\label{thm1}

(i) If $f$ is continuous at any point $x_{0}\in[a,b]$, then%

\[
\lim_{n\rightarrow\infty}D_{n}^{\left(  m\right)  }\left(  f;x_{0}\right)
=f(x_{0}).
\]

(ii) If $f$ $\in C([a,b],[0,1])$, then%

\[
\lim_{n\rightarrow\infty}||D_{n}^{\left(  m\right)  }(f)-f||_{\infty}=0
\]

\end{theorem}

\noindent Here, $||\cdot||_{\infty}$ denotes the supremum norm.

\begin{proof}
Let $x_{0}\in\left[  a,b\right]  $ be a point, where $f$ is continuous. Then
by the triangle inequality we get,%
\[
|D_{n}^{\left(  m\right)  }(f;x_{0})-f(x_{0})|\leq|D_{n}^{\left(  m\right)
}(f,x_{0})-D_{n}^{\left(  m\right)  }(f(x_{0}),x_{0})|+|D_{n}^{\left(
m\right)  }(f(x_{0}),x_{0})-f(x_{0})|
\]
where
\begin{align*}
\left|  D_{n}^{(m)}(f(x_{0}),x_{0}) - f(x_{0}) \right|   &  = \bigg| \bigvee
_{k=na}^{nb-1} \bigg( f(x_{0}) \wedge\dfrac{\phi_{\sigma}(nx_{0}-k)}%
{\bigvee_{d=na}^{nb-1} \phi_{\sigma}(nx_{0}-d)} \bigg) - f(x_{0}) \bigg|\\
&  = \bigg| f(x_{0}) \wedge\bigg( \bigvee_{k=na}^{nb-1} \dfrac{\phi_{\sigma
}(nx_{0}-k)}{\bigvee_{d=na}^{nb-1} \phi_{\sigma}(nx_{0}-d)} \bigg) - f(x_{0})
\bigg|\\
&  =0.
\end{align*}
We also know from the Lemma \ref{lemma5} (d) that
\[
|D_{n}^{\left(  m\right)  }(f,x_{0})-D_{n}^{\left(  m\right)  }(f(x_{0}%
),x_{0})|\leq D_{n}^{\left(  m\right)  }(|f-f(x_{0})|,x_{0}).
\]
Now, since $f$ is continuous at the point $x_{0}$, for every $\varepsilon>0$
there exists a $\delta(x_{0},\varepsilon)>0$ such that $|f(t)-f(x_{0}%
)|<\varepsilon$ whenever $|t-x_{0}| < \delta$. So then,%
\begin{align*}
|D_{n}^{\left(  m\right)  }(f,x_{0})-f(x_{0})|  &  \leq\bigvee\limits_{k\notin
B_{\frac{\delta}{2},n}\left(  x_{0}\right)  }\frac{\int\limits_{a}^{b}
\chi(nt-k)|f(t)-f(x_{0})|dt}{\int\limits_{a}^{b}\chi(nt-k)dt}\wedge\dfrac
{\phi_{\sigma}\left(  nx_{0}-k\right)  }{\bigvee\limits_{d=na} ^{nb-1}%
\phi_{\sigma}\left(  nx_{0}-d\right)  }\\
&  \bigvee\bigvee\limits_{k\in B_{\frac{\delta}{2},n}\left(  x_{0}\right)
}\frac{\int\limits_{a}^{b} \chi(nt-k)|f(t)-f(x_{0})|dt}{\int\limits_{a}%
^{b}\chi(nt-k)dt}\wedge\dfrac{\phi_{\sigma}\left(  nx_{0}-k\right)  }%
{\bigvee\limits_{d=na} ^{nb-1}\phi_{\sigma}\left(  nx_{0}-d\right)  }\\
&  =:T_{1}\vee T_{2}.
\end{align*}
Let us examine $T_{2}$. Since $f\leq1$, from Lemma \ref{lemma2} and Lemma
\ref{lemma0}, there exists a number $K>0$ such that
\begin{align*}
T_{2}  &  \leq\bigvee\limits_{k\in B_{\frac{\delta}{2},n}\left(  x_{0}\right)
}\dfrac{\phi_{\sigma}\left(  nx_{0}-k\right)  }{\bigvee\limits_{d=na}%
^{nb-1}\phi_{\sigma}\left(  nx_{0}-d\right)  }\leq\frac{1}{\phi_{\sigma}%
(2)}\bigvee\limits_{|nx_{0} -k|>\frac{n\delta}{2}}\phi_{\sigma}\left(
nx_{0}-k\right)  \leq\frac{Kn^{-(1+\alpha)}}{\phi_{\sigma}(2)}<\varepsilon
\end{align*}
holds for sufficiently large $n\in\mathbb{N}^{+}.$

Now, let us examine the $T_{1\text{.}}$ Since $a\wedge b\leq a$, from
(\ref{*}) we get,
\begin{align*}
T_{1}  &  \leq\bigvee\limits_{k\notin B_{\frac{\delta}{2},n}\left(
x_{0}\right)  } \frac{\int\limits_{a}^{b}\chi(nt-k)|f(t)-f(x_{0})|dt}%
{\int\limits_{a}^{b} \chi(nt-k)dt}\\
&  \leq\frac{1}{\mathcal{A}}\bigvee\limits_{k\notin B_{\frac{\delta}{2}%
,n}\left(  x_{0}\right)  }n\int\limits_{a}^{b}\chi(nt-k)|f(t)-f(x_{0})|dt.
\end{align*}
Separating the boundary of the interval $[a,b]$, we may obtain
\begin{align*}
T_{1}  &  \leq\frac{1}{\mathcal{A}}\bigg( \bigvee\limits_{k\notin
B_{\frac{\delta}{2},n}\left(  x_{0}\right)  }n\int\limits_{|t-\frac{k}%
{n}|<\frac{\delta}{2}}\chi(nt-k)|f(t)-f(x_{0})|dt\\
&  +\bigvee\limits_{k\notin B_{\frac{\delta}{2},n}\left(  x_{0}\right)  }%
n\int\limits_{|t-\frac{k}{n}|\geq\frac{\delta}{2}}\chi(nt-k)|f(t)-f(x_{0}%
)|dt\bigg) =:\frac{1}{\mathcal{A}}(T_{1,1}+T_{1,2})
\end{align*}
In $T_{1,1}$, since $k\notin B_{\frac{\delta}{2},n}\left(  x_{0}\right)  $ and
$|t-\frac{k}{n}|<\frac{\delta}{2},\text{ we have } |t-x_{0}|\leq|t-\frac{k}%
{n}|+|\frac{k}{n}-x_{0}|<\frac{\delta}{2}+\frac{\delta}{2}=\delta$ and
therefore $|f(t)-f(x_{0})|<\varepsilon.$ Considering this situation together
with $u=nt-k$ substitution, we have%
\[
T_{1,1}<\varepsilon\bigvee\limits_{k\notin B_{\frac{\delta}{2},n}\left(
x_{0}\right)  }\int\limits_{\mathbb{R}}\chi(u)du=\varepsilon\int
\limits_{\mathbb{R}}\chi(u)du=\varepsilon||\chi||_{1}.
\]
In $T_{1,2}$, since $f\leq1$ and $\chi\in L^{1}(\mathbb{R)}$, we have
\[
T_{1,2}\leq\bigvee\limits_{k\notin B_{\frac{\delta}{2},n}\left(  x_{0}\right)
}n\int\limits_{|t-\frac{k}{n}|\geq\frac{\delta}{2}}\chi(nt-k)dt=\int
\limits_{|u|\geq\frac{n\delta}{2}}\chi(u)dt<\varepsilon
\]
holds for sufficiently large $n\in\mathbb{N}$. So in total,%
\[
T_{1}=\frac{1}{\mathcal{A}}(T_{1,1}+T_{1,2})\leq\frac{1}{\mathcal{A}%
}(\varepsilon||\chi||_{1}+\varepsilon)
\]
Then, we finally have
\[
|D_{n}^{\left(  m\right)  }(f,x_{0})-f(x_{0})|\leq T_{1}\vee T_{2}\leq
(\frac{\varepsilon||\chi||_{1}+\varepsilon}{\mathcal{A}})\vee\varepsilon
\]
which completes the first part of the proof since $\varepsilon$ is arbitrary.
For the second part of the theorem, if one assumes $f\in C\left(  \left[
a,b\right]  ,\left[  0,1\right]  \right)  $, then using similar argumentation
lines, and noting that $\delta=\delta\left(  \varepsilon\right)  \,$ the proof
can easily be completed.
\end{proof}

\begin{theorem}
\label{theorem2}Let $f\in C\left(  \left[  a,b\right]  ,\left[  0,1\right]
\right)  .$ Then%
\[
\lim_{n\rightarrow\infty}\left\Vert D_{n}^{\left(  m\right)  }\left(
f\right)  -f\right\Vert _{p}=0.
\]

\end{theorem}

\noindent Here, $\left\Vert \cdot\right\Vert _{p}$ denotes the $L^{p}$ norm on
the interval $\left[  a,b\right]  \ $\ for $1\leq p<\infty$.

\begin{proof}
It is known from the previous theorem that for any $\varepsilon>0$
\[
\left\Vert D_{n}^{\left(  m\right)  }\left(  f\right)  -f\right\Vert _{\infty
}<\varepsilon
\]
for sufficiently large $n\in\mathbb{N}$. Then we obtain
\begin{align*}
\left\Vert D_{n}^{\left(  m\right)  }\left(  f\right)  -f\right\Vert _{p}  &
\leq\bigg( \int\limits_{a}^{b}\left\Vert D_{n}^{\left(  m\right)  }\left(
f\right)  -f\right\Vert _{\infty}^{p}dx\bigg) ^{\frac{1}{p}}=\left\Vert
D_{n}^{\left(  m\right)  }\left(  f\right)  -f\right\Vert _{\infty
}(b-a)^{\frac{1}{p}}\\
&  <\varepsilon(b-a)^{\frac{1}{p}}%
\end{align*}
for sufficiently large $n\in\mathbb{N}$.
\end{proof}

\begin{theorem}
	\label{theorem4}Let $f\in L^{p}\left(  \left[  a,b\right]  ,\left[
	0,1\right]  \right)  $ for $1\leq p<\infty$. Then
	\[
	\lim_{n\rightarrow\infty}\left\Vert D_{n}^{\left(  m\right)  }\left(
	f\right)  -f\right\Vert _{p}=0\text{.}%
	\]
	
\end{theorem}

\begin{proof}
	Assume that $f\in L^{p}\left(  \left[  a,b\right]  ,\left[  0,1\right]
	\right)  $ for $1\leq p<\infty.$ Since $C\left(  \left[  a,b\right]  ,\left[
	0,1\right]  \right)  $ is dense in $L^{p}\left(  \left[  a,b\right]  ,\left[
	0,1\right]  \right)  $, then for all $\varepsilon>0,$ there exists a function
	$h\in C\left(  \left[  a,b\right]  ,\left[  0,1\right]  \right)  $ satisfying that%
	
	\begin{equation}
		\left\Vert f-h\right\Vert _{p}<\varepsilon. \label{***}%
	\end{equation}
	On the other hand, by the Minkowski's inequality we have
	\[
	||D_{n}^{(m)}(f)-f||_{p}\leq||D_{n}^{(m)}(f)-D_{n}^{(m)}(h)||_{p}%
	+||D_{n}^{(m)}(h)-h||_{p}+||f-h||_{p}%
	\]
	Now from (\ref{***}) and from the previous theorem, we have $||D_{n}%
	^{(m)}(h)-h||_{p}<\varepsilon/3$ for sufficiently large $n\in\mathbb{N}$. Now,
	let us concentrate on the first term in the right hand side of the inequality
	above. From Lemma \ref{lemma5} (d) and Lemma \ref{lemma6}, we see that
	\begin{align*}
		I  &  := ||D_{n}^{(m)}(f)-D_{n}^{(m)}(h)||_{p}\\
		&  \leq\bigg(\int\limits_{a}^{b}\left[  D_{n}^{\left(  m\right)
		}(|f-h|;x)\right]  ^{p}dx\bigg)^{\frac{1}{p}}\\
		&  =\bigg(\int\limits_{a}^{b}\bigvee\limits_{k=na}^{nb-1}\bigg[\frac
		{\int\limits_{a}^{b}\chi(nt-k)|f(t)-h(t)|dt}{\int\limits_{a}^{b}\chi
			(nt-k)dt}\bigg]^{p}\wedge\bigg[\dfrac{\phi_{\sigma}\left(  nx-k\right)
		}{\bigvee\limits_{d=na}^{nb-1}\phi_{\sigma}\left(  nx-d\right)  }
		\bigg]^{p}dx\bigg)^{\frac{1}{p}}%
	\end{align*}
	holds. Here, since $\dfrac{\phi_{\sigma}\left(  nx-k\right)  }{\bigvee
		\limits_{d=na}^{nb-1}\phi_{\sigma}\left(  nx-d\right)  }\leq1$, we have
	\[
	\bigg[\dfrac{\phi_{\sigma}\left(  nx-k\right)  }{\bigvee\limits_{d=na}%
		^{nb-1}\phi_{\sigma}\left(  nx-d\right)  }\bigg]^{p}\leq\dfrac{\phi_{\sigma
		}\left(  nx-k\right)  }{\bigvee\limits_{d=na}^{nb-1}\phi_{\sigma}\left(
		nx-d\right)  }.
	\]
	Then as a result of (\ref{*}), we obtain
	\[
	I\leq\bigg(\int\limits_{a}^{b}\bigvee\limits_{k=na}^{nb-1}\bigg[(1/\mathcal{A}%
	)n\int\limits_{a}^{b}\chi(nt-k)|f(t)-h(t)|dt\bigg]^{p}\wedge\dfrac
	{\phi_{\sigma}\left(  nx-k\right)  }{\bigvee\limits_{d=na}^{nb-1}\phi_{\sigma
		}\left(  nx-d\right)  }dx\bigg)^{\frac{1}{p}}.
	\]
	It is not hard to see from Jensen's inequality that {\small
		\[
		\bigg(n\int\limits_{a}^{b}\chi(nt-k)|f(t)-h(t)|dt\bigg)^{p}\leq\bigg(n\int
		\limits_{a}^{b}\chi(nt-k)dt\bigg)^{p-1}\bigg(n\int\limits_{a}^{b}%
		\chi(nt-k)|f(t)-h(t)|^{p}dt\bigg)
		\]
	} and therefore, from (\ref{**}) and Lemma \ref{lemma2}, we get
	\begin{align*}
		I  &  \leq\bigg(\int\limits_{a}^{b}\bigvee\limits_{k=na}^{nb-1}(1/\mathcal{A}%
		^{p})||\chi||_{1}^{p-1}n\int\limits_{a}^{b}\chi(nt-k)|f(t)-h(t)|^{p}%
		dt\wedge\dfrac{\phi_{\sigma}\left(  nx-k\right)  }{\bigvee\limits_{d=na}%
			^{nb-1}\phi_{\sigma}\left(  nx-d\right)  }dx\bigg)^{\frac{1}{p}}\\
		&  \leq\bigg(||\chi||_{1}^{p-1}n\int\limits_{a}^{b}\bigvee\limits_{k=na}%
		^{nb-1}(1/\mathcal{A}^{p})\int\limits_{a}^{b}\chi(nt-k)|f(t)-h(t)|^{p}%
		dt\wedge\dfrac{\phi_{\sigma}\left(  nx-k\right)  }{||\chi||_{1}^{p-1}%
			n\phi_{\sigma}\left(  2\right)  }dx\bigg)^{\frac{1}{p}}.
	\end{align*}
	For brevity, let us set
	\[
	\Gamma_{k}:=\frac{1}{\mathcal{A}^{p}}\int\limits_{a}^{b}\chi
	(nt-k)|f(t)-h(t)|^{p}dt\text{ \ \ }(k=na,\ldots,nb-1),
	\]
	so that the last estimate reads
	\[
	I^{p}\leq||\chi||_{1}^{p-1}n\int\limits_{a}^{b}\bigvee\limits_{k=na}%
	^{nb-1}\bigg(\Gamma_{k}\wedge\dfrac{\phi_{\sigma}\left(  nx-k\right)  }%
	{||\chi||_{1}^{p-1}n\phi_{\sigma}\left(  2\right)  }\bigg)dx.
	\]
	Now, we replace the maximum of nonnegative terms by the corresponding sum,
	\[
	\bigvee\limits_{k=na}^{nb-1}\bigg(\Gamma_{k}\wedge\dfrac{\phi_{\sigma}\left(
		nx-k\right)  }{||\chi||_{1}^{p-1}n\phi_{\sigma}\left(  2\right)  }%
	\bigg)\leq\sum\limits_{k=na}^{nb-1}\bigg(\Gamma_{k}\wedge\dfrac{\phi_{\sigma
		}\left(  nx-k\right)  }{||\chi||_{1}^{p-1}n\phi_{\sigma}\left(  2\right)
	}\bigg),
	\]
	and then, enlarging the domain of integration to $\mathbb{R}$, we get
	\[
	I^{p}\leq||\chi||_{1}^{p-1}n\sum\limits_{k=na}^{nb-1}\int\limits_{\mathbb{R}%
	}\bigg(\Gamma_{k}\wedge\dfrac{\phi_{\sigma}\left(  nx-k\right)  }%
	{||\chi||_{1}^{p-1}n\phi_{\sigma}\left(  2\right)  }\bigg)dx.
	\]
	Now, for each fixed $k$ the substitution $y=nx-k$ is legitimate, and since
	$dx=dy/n$ we obtain
	\[
	I^{p}\leq||\chi||_{1}^{p-1}\sum\limits_{k=na}^{nb-1}\int\limits_{\mathbb{R}%
	}\bigg(\Gamma_{k}\wedge\dfrac{\phi_{\sigma}\left(  y\right)  }{||\chi
		||_{1}^{p-1}n\phi_{\sigma}\left(  2\right)  }\bigg)dy.
	\]
	Due to Lemma \ref{lemma1} (4), there exist $A_{\phi},B>0$ such that
	$\phi_{\sigma}\left(  y\right)  \leq A_{\phi}|y|^{-(1+\alpha)}$ whenever
	$|y|>B$. Setting
	\[
	\tilde{B}:=\max\big\{B,\ ||f-h||_{p}^{-\frac{p}{1+\alpha}}\big\},
	\]
	and splitting the integral according to $|y|>\tilde{B}$ and $|y|\leq\tilde{B}%
	$, we use $u\wedge v\leq v$ on the former and $u\wedge v\leq u$ on the latter:
	\begin{align*}
		I^{p}  &  \leq||\chi||_{1}^{p-1}\sum\limits_{k=na}^{nb-1}\int\limits_{|y|>\tilde
			{B}}\dfrac{\phi_{\sigma}\left(  y\right)  }{||\chi||_{1}^{p-1}n\phi_{\sigma
			}\left(  2\right)  }dy+||\chi||_{1}^{p-1}\sum\limits_{k=na}^{nb-1}%
		\int\limits_{|y|\leq\tilde{B}}\Gamma_{k}dy\\
		&  =:J_{1}+J_{2}.
	\end{align*}
	Since $a,b\in\mathbb{Z}$, the index set $\{na,\ldots,nb-1\}$ consists of
	exactly $n(b-a)$ elements, and hence
	\[
	J_{1}\leq\frac{n(b-a)}{n\phi_{\sigma}\left(  2\right)  }\int\limits_{|y|>\tilde
		{B}}\frac{A_{\phi}}{|y|^{1+\alpha}}dy=\frac{2A_{\phi}(b-a)}{\alpha\phi
		_{\sigma}\left(  2\right)  }\frac{1}{\tilde{B}^{\alpha}},
	\]
	where we used $\int\limits_{|y|>\tilde{B}}|y|^{-(1+\alpha)}dy=\frac{2}%
	{\alpha\tilde{B}^{\alpha}}$. On the other hand, since $\sum\limits_{k\in
		\mathbb{Z}}\chi(nt-k)\leq M_{0}(\chi)$,
	\[
	\sum\limits_{k=na}^{nb-1}\Gamma_{k}\leq\frac{1}{\mathcal{A}^{p}}%
	\int\limits_{a}^{b}|f(t)-h(t)|^{p}\sum\limits_{k\in\mathbb{Z}}\chi
	(nt-k)dt\leq\frac{M_{0}(\chi)}{\mathcal{A}^{p}}||f-h||_{p}^{p},
	\]
	and therefore
	\[
	J_{2}\leq\frac{2\tilde{B}||\chi||_{1}^{p-1}M_{0}(\chi)}{\mathcal{A}^{p}%
	}||f-h||_{p}^{p}.
	\]
	Finally, by the choice of $\tilde{B}$ we have $\tilde{B}^{-\alpha}\leq
	||f-h||_{p}^{\frac{p\alpha}{1+\alpha}}$, and since $|f-h|\leq1$ implies
	$||f-h||_{p}\leq(b-a)^{\frac{1}{p}}$, we also have
	\[
	\tilde{B}||f-h||_{p}^{p}\leq\Lambda||f-h||_{p}^{\frac{p\alpha}{1+\alpha}%
	}\text{, \ \ where }\Lambda:=\max\big\{1,\ B(b-a)^{\frac{1}{1+\alpha}}\big\}.
	\]
	Consequently,
	\begin{align}
		I  &  \leq C_{\ast}||f-h||_{p}^{\frac{\alpha}{1+\alpha}}\text{, \ \ where
		}C_{\ast}:=\bigg(\frac{2A_{\phi}(b-a)}{\alpha\phi_{\sigma}\left(  2\right)
		}+\frac{2\Lambda M_{0}(\chi)}{\mathcal{A}^{p}}||\chi||_{1}^{p-1}%
		\bigg)^{\frac{1}{p}}\label{!!}\\
		&  <C_{\ast}\varepsilon^{\frac{\alpha}{1+\alpha}}.\nonumber
	\end{align}
	Combining the three terms, we finally obtain
	\[
	||D_{n}^{(m)}(f)-f||_{p}\leq C_{\ast}\varepsilon^{\frac{\alpha}{1+\alpha}%
	}+2\varepsilon
	\]
	for sufficiently large $n\in\mathbb{N}$, which completes the proof since
	$\varepsilon>0$ is arbitrary.
\end{proof}
\section{Quantitative Estimates}

Let $f:\left[  a,b\right]  \rightarrow\left[  0,1\right]  $ be given. Then for
a $\delta>0,$ the modulus of continuity of $f$ on $\left[  a,b\right]  $ is
defined as follows :%
\[
\omega_{\left[  a,b\right]  }\left(  f,\delta\right)  :=\sup_{x,y\in\lbrack
a,b]}\left\{  \left\vert f\left(  x\right)  -f\left(  y\right)  \right\vert
:|x-y|\leq\delta\right\}  \text{.}%
\]
Furthermore, we need the definition of generalized absolute moment of order
$\beta>0$, introduced in \cite{cost2} such that for a given $\phi_{\sigma}$,
it is defined as follows
\[
m_{\beta}\left(  \phi_{\sigma}\right)  :=\sup_{x\in\mathbb{R}}\left\{
\bigvee\limits_{k\in\mathbb{Z}}\phi_{\sigma}\left(  x-k\right)  \left\vert
x-k\right\vert ^{\beta}\right\}  \text{.}%
\]

\begin{lemma}
\label{lemf} (see \cite{cost2}) If $0<\beta\leq1+\alpha$, then $m_{\beta
}\left(  \phi_{\sigma}\right)  <\infty$.
\end{lemma}

Now, let us examine the rate of approximation for Theorem \ref{thm1}.

\begin{theorem}
\label{thm5}Let $f\in C\left(  \left[  a,b\right]  ,\left[  0,1\right]
\right)  $. Suppose $\delta_{n}$ and $\bigtriangleup_{n}$ are null sequences
positive real numbers such that $\frac{1}{n.\delta_{n}}$ and $\frac
{1}{n.\bigtriangleup_{n}}$ are also null sequences. Then, the following
inequality holds
\begin{align*}
\left\Vert D_{n}^{\left(  m\right)  }\left(  f\right)  -f\right\Vert
_{\infty}  &  \leq\frac{\omega_{\left[  a,b\right]  }\left(  f,\bigtriangleup
_{n}\right)  }{\mathcal{A}}\left(  ||\chi||_{1}+\frac{1}{n\bigtriangleup_{n}%
}\tilde{M}_{1}(\chi)\right) \\
&  +\left[  \frac{m_{(1+\alpha)}(\phi_{\sigma})}{\phi_{\sigma}(2)}\frac
{1}{(n\delta_{n})^{1+\alpha}}%
{\textstyle\bigvee}
\omega_{\left[  a,b\right]  }\left(  f,\delta_{n}\right)  \right]  \text{.}%
\end{align*}

\end{theorem}

\begin{proof}
Adding and subtracting some suitable terms, we easily obtain
\begin{align*}
|D_{n}^{(m)}(f;x)-f(x)|  &  \leq\bigvee\limits_{k=na}^{nb-1}\frac
{\int\limits_{a}^{b}\chi(nt-k)|f(t)-f(\frac{k}{n})|dt}{\int\limits_{a}^{b}%
\chi(nt-k)dt}\wedge\dfrac{\phi_{\sigma}\left(  nx-k\right)  }{\bigvee
\limits_{d=na}^{nb-1}\phi_{\sigma}\left(  nx-d\right)  }\\
&  +\bigvee\limits_{k=na}^{nb-1}\frac{\int\limits_{a}^{b}\chi(nt-k)|f(\frac
{k}{n})-f(x)|dt}{\int\limits_{a}^{b}\chi(nt-k)dt}\wedge\dfrac{\phi_{\sigma
}\left(  nx-k\right)  }{\bigvee\limits_{d=na}^{nb-1}\phi_{\sigma}\left(
nx-d\right)  }\\
&  =:I_{1}+I_{2}%
\end{align*}
where $I_{2}$ is clearly equal to
\[
I_{2}=\bigvee\limits_{k=na}^{nb-1}|f(\frac{k}{n})-f(x)|\wedge\dfrac
{\phi_{\sigma}\left(  nx-k\right)  }{\bigvee\limits_{d=na}^{nb-1}\phi_{\sigma
}\left(  nx-d\right)  }.
\]
Then seperating $I_{2}$ as follows%
\begin{align*}
I_{2}  &  =\bigg(\bigvee\limits_{k\in B_{\delta_{n},n}(x)}|f(\frac{k}%
{n})-f(x)|\wedge\dfrac{\phi_{\sigma}\left(  nx-k\right)  }{\bigvee
\limits_{d=na}^{nb-1}\phi_{\sigma}\left(  nx-d\right)  }\bigg)\\
&  \bigvee\bigg(\bigvee\limits_{k\notin B_{\delta_{n},n}(x)}|f(\frac{k}%
{n})-f(x)|\wedge\dfrac{\phi_{\sigma}\left(  nx-k\right)  }{\bigvee
\limits_{d=na}^{nb-1}\phi_{\sigma}\left(  nx-d\right)  }\bigg)\\
&  :=I_{2,1}\vee I_{2,2}%
\end{align*}
we get
\begin{align*}
I_{2,2}  &  \leq\bigvee\limits_{k\notin B_{\delta_{n},n}(x)}\omega_{\left[
a,b\right]  }\left(  f,|x-\frac{k}{n}|\right)  \wedge\dfrac{\phi_{\sigma
}\left(  nx-k\right)  }{\bigvee\limits_{d=na}^{nb-1}\phi_{\sigma}\left(
nx-d\right)  }\\
&  \leq\bigvee\limits_{k\notin B_{\delta_{n},n}(x)}\omega_{\left[  a,b\right]
}\left(  f,\delta_{n}\right)  =\omega_{\left[  a,b\right]  }\left(
f,\delta_{n}\right)  \text{.}%
\end{align*}
In $I_{2,1},$ since $k\in B_{\delta_{n},n}(x)$ we have $|x-\frac{k}{n}%
|>\delta_{n}$ which implies $\frac{|nx-k|^{1+\alpha}}{(n.\delta_{n}%
)^{1+\alpha}}>1.$ So,%

\begin{align*}
I_{2,1}  &  \leq\bigvee\limits_{k\in B_{\delta_{n},n}(x)}\dfrac{\phi_{\sigma
}\left(  nx-k\right)  }{\phi_{\sigma}(2)}\\
&  \leq\frac{1}{\phi_{\sigma}(2)}\bigvee\limits_{k\in B_{\delta_{n},n}(x)}%
\phi_{\sigma}\left(  nx-k\right)  \frac{|nx-k|^{1+\alpha}}{(n.\delta
_{n})^{1+\alpha}}\\
&  =\frac{1}{\phi_{\sigma}(2)(n\delta_{n})^{1+\alpha}}\bigvee\limits_{k\in
B_{\delta_{n},n}(x)}\phi_{\sigma}\left(  nx-k\right)  |nx-k|^{1+\alpha}\\
&  \leq\frac{m_{(1+\alpha)}(\phi_{\sigma})}{\phi_{\sigma}(2)}\frac{1}%
{(n\delta_{n})^{1+\alpha}}%
\end{align*}
holds. On the orher hand, in $I$, since
\[
|f(t)-f(\frac{k}{n})|\leq\omega_{\left[  a,b\right]  }\left(  f,|t-\frac{k}%
{n}|\right)  \leq\left(  1+\frac{1}{n\bigtriangleup_{n}}|nt-k|\right)
\omega_{\left[  a,b\right]  }\left(  f,\bigtriangleup_{n}\right)  .
\]
from (\ref{*}) and (\ref{**}) we get
\begin{align*}
I_{1}  &  \leq\frac{\omega_{\left[  a,b\right]  }\left(  f,\bigtriangleup
_{n}\right)  }{\mathcal{A}}\bigg[\bigvee\limits_{k=na}^{nb-1}n\int
\limits_{a}^{b}\chi(nt-k)dt+\frac{1}{n\bigtriangleup_{n}}\bigvee
\limits_{k=na}^{nb-1}n\int\limits_{a}^{b}\chi(nt-k)|nt-k|dt\bigg]\\
&  \leq\frac{\omega_{\left[  a,b\right]  }\left(  f,\bigtriangleup_{n}\right)
}{\mathcal{A}}\big[||\chi||_{1}+\frac{1}{n\bigtriangleup_{n}}\tilde{M}%
_{1}(\chi)\big]\text{.}%
\end{align*}
Then, finally we have
\begin{align*}
\left\Vert D_{n}^{\left(  m\right)  }\left(  f\right)  -f\right\Vert
_{\infty}  &  \leq\frac{\omega_{\left[  a,b\right]  }\left(  f,\bigtriangleup
_{n}\right)  }{\mathcal{A}}\bigg(||\chi||_{1}+\frac{1}{n\bigtriangleup_{n}%
}\tilde{M}_{1}(\chi)\bigg)\\
&  +\bigg[\left(  \frac{m_{(1+\alpha)}(\phi_{\sigma})}{\phi_{\sigma}(2)}%
\frac{1}{(n\delta_{n})^{1+\alpha}}\right)  \bigvee\omega_{\left[  a,b\right]
}\left(  f,\delta_{n}\right)  \bigg].
\end{align*}

\end{proof}

In this part, following \cite{aslan2}, we recall the definition of the
$K$-functionals introduced by Peetre \cite{peetre} and adapted to the max--min
setting in order to derive quantitative estimates for the Durrmeyer-type
max--min neural network operators. For a given $f\in L^{p}\left(  \left[
a,b\right]  ,\left[  0,1\right]  \right)  $ with $1\leq p<\infty$ it is
defined as%

\[
\mathcal{K}\left(  f,\delta\right)  _{p}:=\inf_{h\in C^{1}\left(  \left[
a,b\right]  ,\left[  0,1\right]  \right)  }\left\{  \left\Vert f-h\right\Vert
_{p}^{\frac{\alpha}{\alpha+1}}+\delta\left\Vert h^{\prime}\right\Vert
_{\infty}\right\}
\]
for $\delta>0$. As discussed in \cite{aslan2}, if $\mathcal{K}\left(
f,\delta\right)  _{p}<\varepsilon$ $(\varepsilon>0)$ for sufficiently small
$\delta>0$, which means $f$ can be approximated by a function $h\in
C^{1}\left(  \left[  a,b\right]  ,\left[  0,1\right]  \right)  $ whose
derivative $h^{\prime}$ is not too large. The $K$-functionals thus provide a
quantitative measure of the smoothness in $L^{p}$-spaces. For further
background, the reader is referred to \cite{butzer}.

\begin{theorem}
	\label{thm6}Let $f\in L^{p}([a,b],[0,1])$ for $1\leq p<\infty $. Then
	\[
	\Vert D_{n}^{(m)}(f)-f\Vert _{p}\leq S\,\mathcal{K}\big(f,\,Tn^{-\frac{
			1+\alpha }{2+\alpha }}\big)_{p}+\frac{(b-a)^{\frac{1}{p}}}{n^{\frac{1+\alpha }{
				2+\alpha }}}\cdot \frac{m_{(1+\alpha )}(\phi _{\sigma })}{\phi _{\sigma }(2)},
	\]
	where, with $C_{\ast }$ as in (\ref{!!}),
	\[
	S=C_{\ast }+(b-a)^{\frac{1}{p(1+\alpha )}},\qquad T=\frac{(b-a)^{\frac{1}{p}}}{S}
	\bigg(1+\frac{\tilde{M}_{1}(\chi )}{\mathcal{A}}\bigg).
	\]
\end{theorem}

\begin{proof}
	Let $h\in C^{1}([a,b],[0,1])$ be arbitrary. By the triangle inequality in
	$L^{p}$,
	\[
	\Vert D_{n}^{(m)}(f)-f\Vert _{p}\leq \Vert D_{n}^{(m)}(f)-D_{n}^{(m)}(h)\Vert
	_{p}+\Vert D_{n}^{(m)}(h)-h\Vert _{p}+\Vert f-h\Vert _{p}.
	\]
	By (\ref{!!}), the first term is at most $C_{\ast }\Vert f-h\Vert _{p}^{\alpha
		/(1+\alpha )}$. For the third term, since $0\leq f,h\leq 1$ gives $\Vert
	f-h\Vert _{p}\leq (b-a)^{1/p}$,
	\[
	\Vert f-h\Vert _{p}=\Vert f-h\Vert _{p}^{\frac{\alpha }{1+\alpha }}\Vert
	f-h\Vert _{p}^{\frac{1}{1+\alpha }}\leq (b-a)^{\frac{1}{p(1+\alpha )}}\Vert
	f-h\Vert _{p}^{\frac{\alpha }{1+\alpha }},
	\]
	so that the first and the third term together are bounded by $S\Vert f-h\Vert
	_{p}^{\alpha /(1+\alpha )}$.
	
	We now estimate the middle term. Let $\delta _{n}>0$ and $x\in \lbrack a,b]$.
	Since $D_{n}^{(m)}(h(x);x)=h(x)$, Lemma \ref{lemma5} (d) gives
	$|D_{n}^{(m)}(h;x)-h(x)|\leq D_{n}^{(m)}(|h-h(x)|;x)$, and by the mean value
	theorem $|h(t)-h(x)|\leq \Vert h^{\prime }\Vert _{\infty }|t-x|\leq \Vert
	h^{\prime }\Vert _{\infty }\big(|x-\tfrac{k}{n}|+|\tfrac{k}{n}-t|\big)$. Hence
	\begin{align*}
		|D_{n}^{(m)}(h;x)-h(x)| &\leq \bigvee\limits_{k=na}^{nb-1}\bigg(\Vert h^{\prime
		}\Vert _{\infty }\frac{\int_{a}^{b}\chi (nt-k)|x-\frac{k}{n}|dt}{\int_{a}^{b}\chi
			(nt-k)dt}\wedge \frac{\phi _{\sigma }(nx-k)}{\bigvee_{d=na}^{nb-1}\phi _{\sigma
			}(nx-d)}\bigg)\\
		&+\bigvee\limits_{k=na}^{nb-1}\bigg(\Vert h^{\prime }\Vert _{\infty }\frac{
			\int_{a}^{b}\chi (nt-k)|\frac{k}{n}-t|dt}{\int_{a}^{b}\chi (nt-k)dt}\wedge \frac{
			\phi _{\sigma }(nx-k)}{\bigvee_{d=na}^{nb-1}\phi _{\sigma }(nx-d)}\bigg)=:V_{1}+V_{2}.
	\end{align*}
	
	For $V_{2}$, using $u\wedge v\leq u$, the lower bound (\ref{*}), enlarging the
	domain of integration to $\mathbb{R}$ and substituting $u=nt-k$, we obtain
	\[
	V_{2}\leq \frac{\Vert h^{\prime }\Vert _{\infty }}{n\mathcal{A}}\bigvee
	\limits_{k=na}^{nb-1}n\int\limits_{\mathbb{R}}\chi (nt-k)|nt-k|\,dt=\frac{\Vert
		h^{\prime }\Vert _{\infty }\tilde{M}_{1}(\chi )}{n\mathcal{A}}.
	\]
	
	For $V_{1}$ we split the index set, so that $V_{1}=V_{1,1}\vee V_{1,2}$, where
	$V_{1,1}$ and $V_{1,2}$ denote the maxima over $k\in B_{\delta _{n},n}(x)$ and
	over $k\notin B_{\delta _{n},n}(x)$, respectively. If $k\notin B_{\delta
		_{n},n}(x)$, then $|x-\frac{k}{n}|\leq \delta _{n}$, whence, using $u\wedge
	v\leq u$,
	\[
	V_{1,2}\leq \bigvee\limits_{k\notin B_{\delta _{n},n}(x)}\Vert h^{\prime }\Vert
	_{\infty }\frac{\int_{a}^{b}\chi (nt-k)\delta _{n}dt}{\int_{a}^{b}\chi (nt-k)dt}
	=\Vert h^{\prime }\Vert _{\infty }\delta _{n}.
	\]
	If $k\in B_{\delta _{n},n}(x)$, then $|x-\frac{k}{n}|>\delta _{n}$ and
	consequently $|nx-k|^{1+\alpha }/(n\delta _{n})^{1+\alpha }>1$, so that, using
	$u\wedge v\leq v$ together with Lemma \ref{lemma2},
	\[
	V_{1,1}\leq \bigvee\limits_{k\in B_{\delta _{n},n}(x)}\frac{\phi _{\sigma }(nx-k)
	}{\phi _{\sigma }(2)}\leq \frac{1}{\phi _{\sigma }(2)(n\delta _{n})^{1+\alpha }}
	\bigvee\limits_{k\in B_{\delta _{n},n}(x)}\phi _{\sigma }(nx-k)|nx-k|^{1+\alpha
	}\leq \frac{m_{(1+\alpha )}(\phi _{\sigma })}{\phi _{\sigma }(2)(n\delta
		_{n})^{1+\alpha }}.
	\]
	Combining the two cases,
	\[
	V_{1}=V_{1,1}\vee V_{1,2}\leq \Vert h^{\prime }\Vert _{\infty }\delta _{n}\vee 
	\frac{m_{(1+\alpha )}(\phi _{\sigma })}{\phi _{\sigma }(2)}\frac{1}{(n\delta
		_{n})^{1+\alpha }},
	\]
	and therefore
	\[
	|D_{n}^{(m)}(h;x)-h(x)|\leq V_{1}+V_{2}\leq \bigg(\Vert h^{\prime }\Vert
	_{\infty }\delta _{n}\vee \frac{m_{(1+\alpha )}(\phi _{\sigma })}{\phi _{\sigma
		}(2)}\frac{1}{(n\delta _{n})^{1+\alpha }}\bigg)+\frac{\Vert h^{\prime }\Vert
		_{\infty }\tilde{M}_{1}(\chi )}{n\mathcal{A}}
	\]
	for every $\delta _{n}>0$. Choosing $\delta _{n}=n^{-\frac{1+\alpha }{2+\alpha }}
	$, so that $(n\delta _{n})^{-(1+\alpha )}=\delta _{n}$, and using $n^{-1}\leq
	\delta _{n}$ together with $u\vee v\leq u+v$, we get
	\[
	|D_{n}^{(m)}(h;x)-h(x)|\leq \delta _{n}\Vert h^{\prime }\Vert _{\infty }\bigg(1+
	\frac{\tilde{M}_{1}(\chi )}{\mathcal{A}}\bigg)+\delta _{n}\frac{m_{(1+\alpha
			)}(\phi _{\sigma })}{\phi _{\sigma }(2)},
	\]
	uniformly in $x\in \lbrack a,b]$. Hence
	\[
	\Vert D_{n}^{(m)}(h)-h\Vert _{p}\leq (b-a)^{\frac{1}{p}}\delta _{n}\bigg[\Vert
	h^{\prime }\Vert _{\infty }\bigg(1+\frac{\tilde{M}_{1}(\chi )}{\mathcal{A}}
	\bigg)+\frac{m_{(1+\alpha )}(\phi _{\sigma })}{\phi _{\sigma }(2)}\bigg].
	\]
	
	Combining the three estimates and recalling the definitions of $S$ and $T$,
	\[
	\Vert D_{n}^{(m)}(f)-f\Vert _{p}\leq S\Big[\Vert f-h\Vert _{p}^{\frac{\alpha }{
			1+\alpha }}+T\delta _{n}\Vert h^{\prime }\Vert _{\infty }\Big]+(b-a)^{\frac{1}{p}
	}\delta _{n}\frac{m_{(1+\alpha )}(\phi _{\sigma })}{\phi _{\sigma }(2)}.
	\]
	The last term does not depend on $h$, while $h\in C^{1}([a,b],[0,1])$ was
	arbitrary. Taking the infimum over all such $h$ in the bracket yields
	$\mathcal{K}(f,T\delta _{n})_{p}$ with $\delta _{n}=n^{-\frac{1+\alpha }{
			2+\alpha }}$, which is the assertion.
\end{proof}

\section{Applications}

\label{sec5}

In this section, firtly, we present common sigmoidal function examples that satisfy the
assumptions of our theory and then we examine the application areas of Durrmeyer-type
max-min neural network operators. We
also evaluate the performance of these operators on corrupted data. In addition, we present the approximation error of the
operator, comparing it with the existing operators mentioned in the
introduction of the study. Now let us introduce some well-known sigmoidal
functions below.

The logistic and hyperbolic tangent functions are two well-known examples of smooth sigmoidal functions that serve as non-compactly supported kernels $\phi_{\sigma}$
\[
\sigma_{l}:=\frac{1}{1+e^{-x}}\text{ and }\sigma_{h}:=\frac{1}{2}\left(  \tanh
x+1\right)  \text{ \ \ }\left(  x\in\mathbb{R}\right)
\]
respectively. Furthermore, these sigmoidal functions satisfy the $\left(
\Sigma3\right)  $ condition for every $\alpha>0$. Examples of non-smooth
sigmoidal functions with a compact support kernel ($\phi_{\sigma}$) structure
are the discontinuous step function (see \cite{cost2}) and the ramp function
(see \cite{cao2, cheang}) respectively, defined in the context of Remark \ref{rem0}. 
{\small\[
\sigma_{three}:=\left\{
\begin{array}
[c]{ll}%
0, & \text{if }x<-1/2\\
1/2, & \text{if }-1/2\leq x\leq1/2\\
1, & \text{if }x>1/2
\end{array}
\right.  \text{ and }\sigma_{R}:=\left\{
\begin{array}
[c]{ll}%
0, & \text{if }x<-1/2\\
x+1/2, & \text{if }-1/2\leq x\leq1/2\\
1, & \text{if }x>1/2
\end{array}
\right.
\]}
It can be easily
observed that both functions ($\sigma_{three}$ and $\sigma_{R}$) satisfy the
$\left(  \Sigma3\right)  $ condition for every $\alpha>0$. 

Now, we observe the results of our operator $D_{n}^{\left(  m\right)  }$ in
the applications and compare them with other known NN operators. The
discontinuous function $f:\left[  0,1\right]  \rightarrow\left[  0,1\right]  $
used for comparison is defined as follows:%
\[
f\left(  x\right)  :=\left\{
\begin{array}
[c]{ll}%
0.25, & \text{if }0\leq x\leq0.15\\
0.72, & \text{if }0.15<x\leq0.4\\
0.35, & \text{if }0.40<x\leq0.7\\
0.55, & \text{if }0.70<x\leq1
\end{array}
\right.  \text{.}%
\]
First, we aim to analyze the approximation behavior of the operators $F_{n}^{(m)}$, $K_{n}^{(m)}$, and the Durrmeyer-type max-min NN operator $D_{n}^{(m)}$ for this discontinuous function $f$. If we take $n=200$, $\sigma=\sigma
_{l}\,\ $and $\chi(x)=\frac{1}{1+x^{2}}$, we obtain the result shown in Figure
\ref{fig1}. 
\begin{figure}[h]
\centering
\includegraphics[width=.9\linewidth]{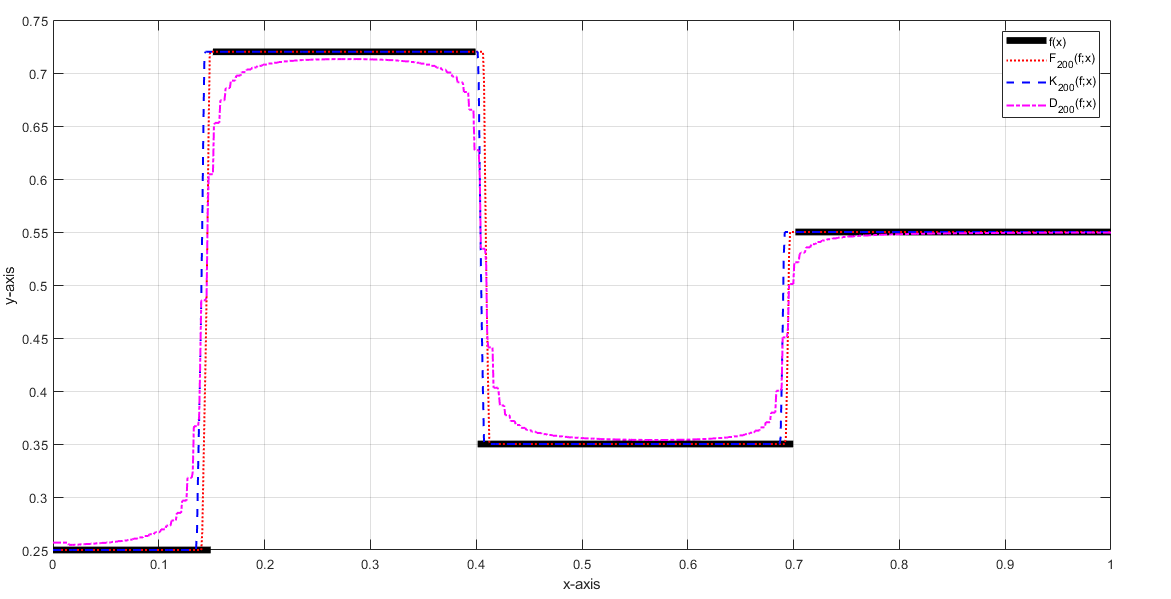} \caption{Approximations by
$D_{n}^{(m)}(f)$, $F_{n}^{(m)}(f)$ and $K_{n}^{(m)}(f)$.}%
\label{fig1}%
\end{figure}

\noindent From the graph in Figure \ref{fig1}, it can be observed that all
operators generally capture the correct values in regions close to the
constant regions of the piecewise function $f$, but that the behavior of the
$D_{n}^{\left(  m\right)  }$ operator is significantly smoother than standart max-min
and Kantorovich max-min NN operators. It was observed that the $F_{n}^{(m)}$
and $K_{n}^{(m)}$ operators produced relatively sharper transitions,
especially at critical points, in contrast, $D_{n}^{\left(  m\right)  }$
exhibited wider ramps and smoother corner transitions at the same points. This
situation can be interpreted as an indication that $D_{n}^{\left(  m\right)
}$ provides a weighted-average-based approach rather than a point-based one.
In Durrmeyer-type max-min operators, weighted integrals are taken for each
cell using $\chi$, this process naturally creates a local averaging/filtering effect.

Additionally we would like to introduce some error types that we obtained from
these operators. These are, in order, maximum error (ME), mean absolute error
(MAE), and mean squared error (MSE). In this process, $8000$ sampling point is taken  from $[0,1]$ interval. The error table (Table \ref{T1}) obtained
for the values mentioned above is provided below. 

\begin{table}[h]
\caption{Comparison of Approximation Errors}%
\label{T1}
\centering
\begin{tabular}
[c]{c|c|c|c|c}%
\textbf{NN operators} & \textbf{ME} & \textbf{MAE} & \textbf{MSE} & \\\hline
$D_{200}^{(m)}(f)$ & 0.354421 & 0.017836 & 0.001928 & \\
$K_{200}^{(m)}(f)$ & 0.470000 & 0.008435 & 0.002813 & \\
$F_{200}^{(m)}(f)$ & 0.470000 & 0.006937 & 0.002194 & 
\end{tabular}
\end{table}

When the provided error measures are examined, it becomes evident that there
are notable differences among the approximation performances of the
$D_{n}^{\left(  m\right)  }$, $K_{n}^{(m)}$ and
$F_{n}^{(m)}$. An analysis of the maximum errors shows that
$D_{n}^{\left(  m\right)  }$ operator produces a lower worst-case deviation
compared to the other two operators, indicating a more balanced behavior that
avoids large spikes in the error distribution. However, when the mean absolute
error is considered, the $F_{n}^{(m)}$
operator yields the most accurate results on average, suggesting that although
it may occasionally produce larger deviations, it generally maintains a lower
mean error across the domain. The mean squared error values (MSE) further
reveal that the $D_{n}^{\left(  m\right)  }$ operator performs best with
respect to this metric, reflecting a more stable error structure under a
criterion that penalizes larger deviations more heavily. In contrast, the
$K_{n}^{(m)}$ operator demonstrates weaker performance in terms of both
maximum error and MSE, although it achieves a reasonable average error level,
it appears more sensitive to the presence of large deviations. Overall, these
findings indicate that the $F_{n}^{(m)}$ operator is preferable in
applications where average accuracy is prioritized, whereas the $D_{n}%
^{\left(  m\right)  }$ operator provides a more balanced and reliable
performance in scenarios where minimizing worst-case errors or squared
deviations is essential. The $K_{n}^{(m)}$ operator occupies an intermediate
position between these two behaviors but remains comparatively more vulnerable
when large errors dominate the approximation process.

\vspace{2mm}

\textbf{\noindent Comparative Analysis of Denoising Performance of Max-Min\\
Durrmeyer-Type Neural Network Operators.} Previous studies have demonstrated
that max--min type neural network operators outperform their max--product
counterparts in various filtering tasks (see \cite{is}, \cite{aslan2}).
Building on these findings, this section provides a comprehensive performance
evaluation of max--min based neural network operators, specifically the
sampling max-min kind ($F_{n}^{(m)}$), Kantorovich-type ($K_{n}^{(m)}$) and
Durrmeyer-type ($D_{n}^{\left(  m\right)  }$) formulations in the context of
signal denoising under different noise conditions. The comparative analysis
investigates their filtering capabilities against both Gaussian and salt and
pepper noise, supported by visual and numerical assessments. First, we introduce the function $g$, to which noise will be added.
\[
g(x)=0.45+0.25\sin(8\pi x)
\]
 Next, max-min NN operators are applied to filter the noisy function. For salt and pepper noise scenarios, the experimental setup employed $n=8000$
sampling points with $\chi(x)=1/(1+0.001x^{2})$ and $\phi
_{\sigma_{h}}(x)=\frac{1}{2}\left(  \sigma_{h}(0.05(x+1))-\sigma
_{h}(0.05(x-1))\right)  $. The results obtained are shown in Figure \ref{figg2}.
\begin{figure}[h]
\centering
\includegraphics[width=1\linewidth]{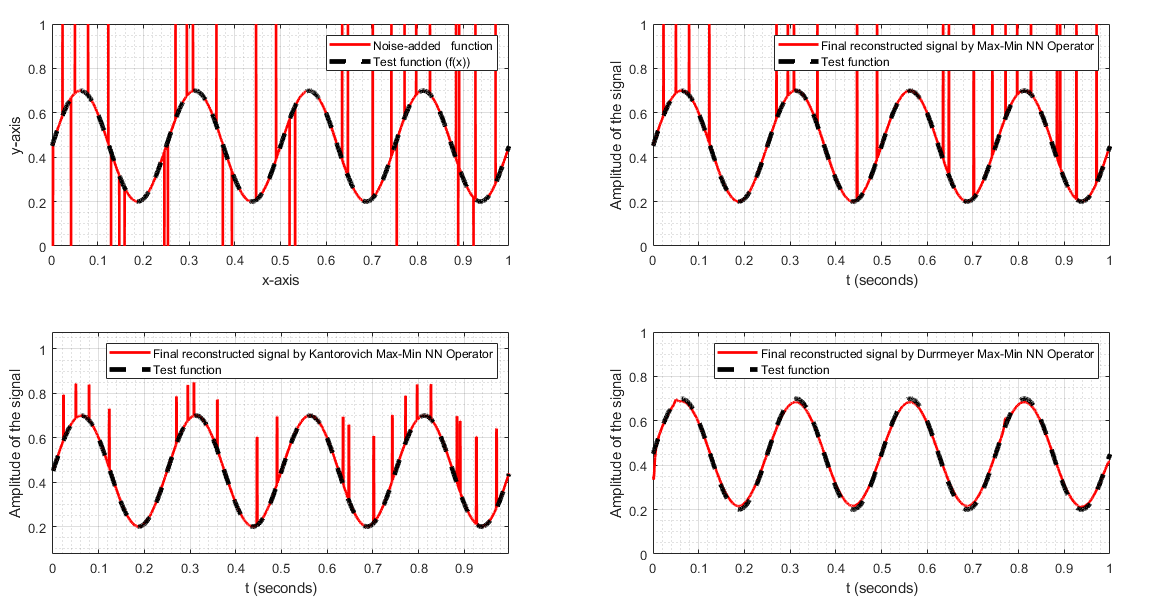}\caption{Comparison of the filtering capabilities of max-min type NN operators on signals corrupted by salt and pepper noise.}%
\label{figg2}%
\end{figure}
It is observed that when the max-min NN operator and the
Kantorovich max-min NN operators are applied once to the noisy
function, denoted by $\dddot{g}$, pepper-type noise is effectively removed, whereas salt-type
noise cannot be completely filtered. On the other hand,
Durrmeyer-type max-min NN operator ($D_{n}^{(m)}$) is capable of suppressing
both types of noise. Maximum approximation errors
obtained from these approximations are provided in Table \ref{T4}. As shown in the table, $D_{n}^{(m)}$ has the lowest
maximum error, indicating that it is the most effective operator for filtering
salt and pepper noise. $K_{n}^{(m)}$ and $F_{n}^{(m)}$ show significantly
weaker performances. However, this limitation can be overcome
with a small methodological adjustment (see also \cite{is}). Specifically, let $h$ denote the function obtained after the first step, where the pepper noise is eliminated by the corresponding operator $F_{n}^{(m)}$ or $K_{n}^{(m)}$. Subsequently, applying the $F_{n}^{(m)}$ or $K_{n}^{(m)}$ operators to the function $1-h$ enables the removal of the remaining salt noise. 
\begin{table}[h]
	\caption{Comparison of Maximum Errors for Salt and Pepper Filtering}%
	\label{T4}
	\centering
	\begin{tabular}
		[c]{c|c}%
		\textbf{NN operator} & \textbf{Maximum Error}\\\hline
		$D_{8000}^{(m)}(\dddot{g})$ & 0.063663\\
		$K_{8000}^{(m)}(\dddot{g})$ & 0.748019\\
		$F_{8000}^{(m)}(\dddot{g})$ & 0.797163
	\end{tabular}
\end{table} However, since $D_{n}^{(m)}$ filtered the noisy signal simultaneously in a single approximation step, the overall
computational time is significantly reduced. Even when the computational
complexity introduced by the $\chi$ function is taken into account, a
substantial decrease in processing time is still achieved. The processing
times reported in Table \ref{T2} clearly confirm this improvement.
\begin{table}[h]
\caption{Comparison of max-min type NN Operators' Computational Complexity}%
\label{T2}
\centering
\begin{tabular}
[c]{c|c}%
\textbf{NN operator} & \textbf{Computation Time}\\\hline
$D_{8000}^{(m)}(\dddot{g})$ & 1.949426 seconds\\
$1-K_{8000}^{(m)}(1-K_{8000}(\dddot{g}))$ & 4.265475 seconds\\
$1-F_{8000}^{(m)}(1-F_{8000}(\dddot{g}))$ & 3.457160 seconds
\end{tabular}
\end{table}We would like to note that the calculation times may vary depending
on the hardware power of the computer used, but in any case, the calculation
time for the $D_{n}^{(m)}$ operator will be shorter than that of the other two operators.

Now let us consider the Gaussian noise situation. In this situation the
experimental setup employed $n=8000$ sampling points with $\chi
(x)=1/(1+0.002x^{2})$, a rescaled logistic sigmoidal activation function
($\sigma(x)=1/(1+e^{-10x}$) and $\phi_{\sigma}(x)=\frac{1}{2}\left(
\sigma(0.05(x+1))-\sigma(0.05(x-1))\right)  $ was used. The results obtained
are shown in Figure \ref{3fig}.
\begin{figure}[h]
\centering
\includegraphics[width=1\linewidth]{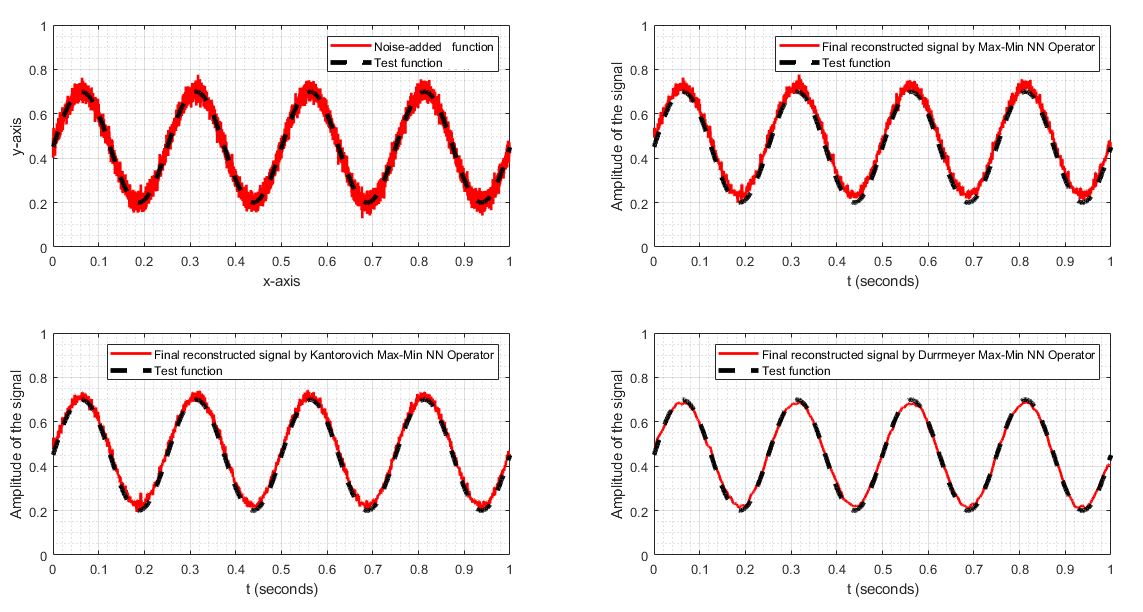}\caption{Comparison of the filtering capabilities of max-min type NN operators on signals corrupted by Gaussian noise.}%
\label{3fig}
\end{figure}
The presented graphs clearly illustrate the filtering performance of the
max-min NN, Kantorovich max-min NN, and Durrmeyer max-min NN operators on a signal
corrupted solely by Gaussian noise. Examination of the noisy function
$\tilde{g}$ shows that the signal contains dense random fluctuations, with
high-frequency oscillations constituting the dominant form of distortion. When
the max-min NN operator is applied, the signal becomes noticeably smoother
and generally approaches the test function, however, a complete suppression of
noise is not achieved, and oscillatory deviations, particularly around peak
regions, remain visible. The Kantorovich max-min NN operator provides a more
balanced improvement compared to the max-min version, as its
averaging-based structure results in a smoother reconstructed signal overall.
Among the three methods, the Durrmeyer max-min operator clearly yields the
best performance. The output produced by this operator corresponds to the test
function with very high accuracy in terms of both amplitude (the magnitude of
the signal) and phase (the temporal position or horizontal shift of the
signal). The remarkable preservation of the waveform both at peak locations
and transition regions demonstrates the superior capability of the Durrmeyer
operator in filtering Gaussian noise. Overall, the results show that while the
max-min NN and Kantorovich max-min NN operators are effective in reducing
Gaussian noise, the Durrmeyer form is the method that provides the most
accurate and most stable reconstruction in a single step. The error approaches
obtained additionally are also provided in Table \ref{T5}. \begin{table}[h]
\caption{Comparison of Approximation Errors in Filtering Gaussian Noise}%
\label{T5}
\centering
\begin{tabular}
[c]{c|c|c|c}%
\textbf{NN operators} & \textbf{ME} & \textbf{MAE} & \textbf{MSE}\\\hline
$D_{8000}^{(m)}(\tilde{g})$ & 0.043498 & 0.011138 & 0.000167\\
$K_{8000}^{(m)}(\tilde{g})$ & 0.078481 & 0.020367 & 0.000533\\
$F_{8000}^{(m)}(\tilde{g})$ & 0.082371 & 0.031181 & 0.001178\\
\end{tabular}
\end{table}As is clearly evident, the $D_{n}^{(m)}$ operator is also clearly
the most advantageous in Gaussian noise filtering.

In the final part of our study, the noise reduction performance of
Durrmeyer-type max--min neural network operators on speech signals is
investigated through an experiment. For experimental evaluation, the Turkish
word \textquotedblleft merhaba\textquotedblright, which means "hello" in
English, was recorded, and the resulting raw speech signal was imported into
the MATLAB environment. Initially, the recorded original speech signal was
considered as the reference signal. Subsequently, salt and pepper noise was
added to the signal in order to represent realistic and abrupt distortions. Such noise types pose significant
challenges for classical linear filtering methods.

The noise-contaminated speech signal was then processed using the proposed
Durrmeyer-type max-min NN operator. Owing to its
nonlinear structure, this operator provides an effective approach against
impulsive noise and aims to suppress noise while preserving the fundamental
time--amplitude characteristics of the speech signal.
\begin{figure}[ptbh]
\centering
\includegraphics[width=1\linewidth]{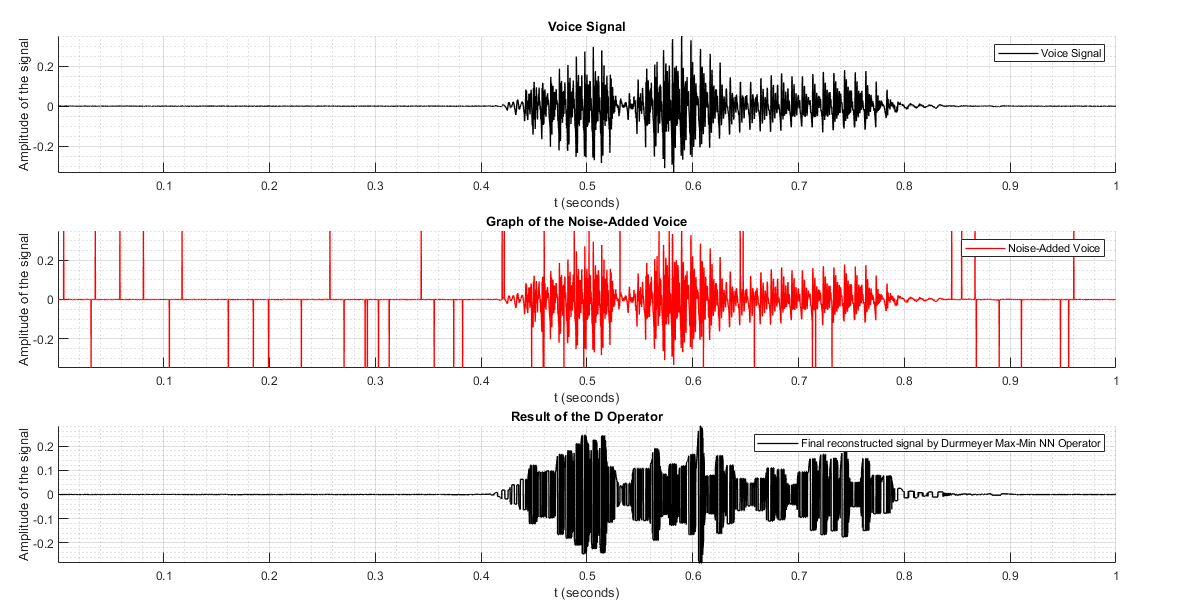} \caption{Recorded speech
and filtering result. Top: Original recorded waveform of the Turkish word
\textquotedblleft merhaba\textquotedblright. Middle: The same signal after
contamination with salt and pepper noise. Bottom: Reconstructed (filtered)
signal obtained by applying the Durrmeyer-type maximum-minimum neural network
operator.}%
\label{fig:tuzbiber}%
\end{figure}
Based on the provided Figure \ref{fig:tuzbiber}, it can be observed that the
proposed operator largely suppresses noise-induced abrupt amplitude
distortions, while successfully preserving the time-amplitude structure of
the speech signal. In particular, signal continuity is maintained during time
intervals where speech activity is present, whereas impulsive noise components
occurring in silent regions are effectively eliminated.

These experimental results indicate that Durrmeyer-type maximum--minimum
neural network operators are an effective tool for removing noise, and offer a
successful filtering approach for speech signal processing applications.

\section{Concluding Remarks}

This research presents a comprehensive investigation into Durrmeyer-type max-min neural network operators, focusing particularly on their convergence properties and practical applications. The theoretical
research is supported by numerical experiments that validate established
analytical findings and demonstrate some of the practical benefits of these
operators in noise reduction applications for corrupted signals. In
conclusion, Durrmeyer-type max-min NN operators provide significantly improved
noise reduction performance compared to both Kantorovich-type max-min NN and standart
max-min NN operators applications. Since many signals, such
as images, videos etc. are multivariable, our next aim is to study
multivariate case of Durrmeyer type max-min NN operators.

\section*{Acknowledgements}

The authors are funded by the Scientific and Technological Research Council of
T\"{u}rkiye (T\"{U}B\.{I}TAK) under grant no. 124F356. This study is derived from the Master's thesis of Berke \c{S}ahin.

\medskip

\textbf{Berke \c{S}ahin}

Hacettepe University

Department of Mathematics,

\c{C}ankaya TR-06800, Ankara, T\"{u}rkiye

E-mail: berke.sahin@hacettepe.edu.tr

\bigskip

\textbf{\.{I}smail Aslan}

Hacettepe University

Department of Mathematics,

\c{C}ankaya TR-06800, Ankara, T\"{u}rkiye

E-mail: ismail-aslan@hacettepe.edu.tr

\bigskip

\bigskip


\begin{thebibliography}{99}                                                                                               %
\bibitem {anas7}Anastassiou, G. A. (2019). Approximation by multivariate
sublinear and max-product operators. Revista de la Real Academia de Ciencias
Exactas, F\'{\i}sicas y Naturales. Serie A. Matem\'{a}ticas, 113, 507-540.

\bibitem {is}Aslan, I. (2025). Approximation by Max-Min Neural Network
Operators. Numerical Functional Analysis and Optimization, 46(4-5), 374-393.

\bibitem {aslan2}Aslan, \.{I}., De Marchi, S., \& Erb, W. (2024). $L^{p}%
$-convergence of Kantorovich-type Max-Min Neural Network Operators. arXiv
preprint arXiv:2407.03329.

\bibitem {bede3}Bede, B., Coroianu, L., \& Gal, S. G. (2009). Approximation
and shape preserving properties of the Bernstein operator of max-product kind.
International journal of mathematics and mathematical sciences, 2009.

\bibitem {bede4}Bede, B., Coroianu, L., \& Gal, S. G. (2010). Approximation
and shape preserving properties of the nonlinear Favard-Szasz-Mirakjan
operator of max-product kind. Filomat, 24(3), 55-72.

\bibitem {bede5}Bede, B., Coroianu, L., \& Gal, S. G. (2016). Approximation by
max-product type operators. Heidelberg: Springer.

\bibitem {bede1}Bede, B., Nobuhara, H., Da\v{n}kov\'{a}, M., \& Di Nola, A.
(2008). Approximation by pseudo-linear operators. Fuzzy Sets and Systems,
159(7), 804-820.

\bibitem {bede2}Bede, B., Schwab, E. D., Nobuhara, H., \& Rudas, I. J. (2009).
Approximation by Shepard type pseudo-linear operators and applications to
image processing. International journal of approximate reasoning, 50(1), 21-36.

\bibitem {butzer}Butzer, P. L., \& Berens, H. (2013). Semi-groups of operators
and approximation (Vol. 145). Springer Science \& Business Media.

\bibitem {durrmeyer}Coroianu, L., Costarelli, D., Natale, M., \& Panti\c{s},
A. (2024). The approximation capabilities of Durrmeyer-type neural network
operators. Journal of Applied Mathematics and Computing, 70(5), 4581-4599.

\bibitem {costez}Costarelli, D, (2014). Sigmoidal functions approximation and
applications (Ph.D. thesis), Roma Tre University, Rome, Italy.


\bibitem {cost1}Costarelli, D., \& Spigler, R. (2013). Approximation results
for neural network operators activated by sigmoidal functions. Neural
Networks, 44, 101-106.

\bibitem {cost2}Costarelli, D., \& Vinti, G. (2016). Max-product neural
network and quasi-interpolation operators activated by sigmoidal functions.
Journal of Approximation Theory, 209, 1-22.

\bibitem {cost25}Costarelli, D., \& Vinti, G. (2016). Approximation by
max-product neural network operators of Kantorovich type. Results in
Mathematics, 69, 505-519.

\bibitem {carda}Cardaliaguet, P., \& Euvrard, G. (1992). Approximation of a
function and its derivative with a neural network. Neural networks, 5(2), 207-220.

\bibitem {cheang}Cheang, G. H. (2010). Approximation with neural networks
activated by ramp sigmoids. Journal of Approximation Theory, 162(8), 1450-1465.

\bibitem {cao}Chen, Z., \& Cao, F. (2009). The approximation operators with
sigmoidal functions. Computers \& Mathematics with Applications, 58(4), 758-765.

\bibitem {cao2}Chen, Z., \& Cao, F. (2012). The construction and approximation
of a class of neural networks operators with ramp functions. Journal of
Computational Analysis and Applications, 14(1), 101-112.

\bibitem {coroianu4}Coroianu, L., Costarelli, D., \& Kadak, U. (2022).
Quantitative estimates for neural network operators implied by the asymptotic
behaviour of the sigmoidal activation functions. Mediterranean Journal of
Mathematics, 19(5), 211.

\bibitem {coroianu1}Coroianu, L., Costarelli, D., Gal, S. G., \& Vinti, G.
(2019). The max-product generalized sampling operators: convergence and
quantitative estimates. Applied Mathematics and Computation, 355, 173-183.

\bibitem {coroianu2}Coroianu, L., \& Gal, S. G. (2018). Approximation by
truncated max-product operators of Kantorovich-type based on generalized
($\Phi,\Psi$)-kernels. Mathematical Methods in the Applied Sciences, 41(17), 7971-7984.

\bibitem {coroianu3}Coroianu, L., \& Gal, S. G. (2022). New approximation
properties of the Bernstein max-min operators and Bernstein max-product
operators. Mathematical Foundations of Computing, 5(3), 259-268.

\bibitem {duman}Duman, O. (2024). Max-product Shepard operators based on
multivariable Taylor polynomials. Journal of Computational and Applied
Mathematics, 437, 115456.

\bibitem {aslan}G\"{o}k\c{c}er, T. Y., \& Aslan, \.{I}. (2022). Approximation
by Kantorovich-type max-min operators and its applications. Applied
Mathematics and Computation, 423, 127011.

\bibitem {gokcer2}G\"{o}k\c{c}er, T. Y., \& Duman, O. (2020). Approximation by
max-min operators: A general theory and its applications. Fuzzy Sets and
Systems, 394, 146-161.

\bibitem {holhos1}Holho\c{s}, A. (2018). Weighted approximation of functions
by Meyer--K\"{o}nig and Zeller operators of max-product type. Numerical
Functional Analysis and Optimization, 39(6), 689-703.

\bibitem {kad}Kadak, U. (2022). Multivariate neural network interpolation
operators. Journal of Computational and Applied Mathematics, 414, 114426.

\bibitem {peetre}Peetre, J. (1970). A new approach in interpolation spaces.
Studia Mathematica, 34(1), 23-42.

\bibitem {wu}Wu, Y., and Yu, D. (2025). Approximation by weighted
Durrmeyer-type max-product neural network operators. Demonstratio Mathematica,
58(1), 20250211.
\end{thebibliography}
\end{document}